\documentclass{amsart}
\usepackage{enumerate}
\usepackage{rotating}
\usepackage{amsfonts}
\usepackage{amssymb}
\usepackage{euscript}
\usepackage[bookmarks,colorlinks=true]{hyperref}

\usepackage{epic}
\usepackage[dvips]{color}

\newcommand{\NN}{{\mathbb{N}}}  

\newcommand{\ZZ}{\mathbb{Z}}

\newcommand{\QQ}{\mathbb{Q}}

\renewcommand{\phi}{\varphi}

\renewcommand{\lim}{\mbox{lim}}

\newcommand{\fa}{\forall}

\renewcommand{\tilde}{\widetilde}

\newtheorem{theorem}{Theorem}[section]

\newtheorem{definability lemma}[theorem]{Definability Lemma}

\newtheorem{proposition}[theorem]{Proposition}

\newtheorem{definition}[theorem]{Definition}

\newtheorem{lemma}[theorem]{Lemma}
\newtheorem{corollary}[theorem]{Corollary}

\newtheorem{question}[theorem]{Question}

\newlength{\tri}
\newcommand{\bq}{\begin{quotation}\setlength{\parindent}{0in} }
\newcommand{\eq}{\end{quotation}}

\newcommand{\be}{\begin{enumerate}}
\newcommand{\ee}{\end{enumerate}}

\newcommand{\bee}{\begin{enumerate}}
\newcommand{\eee}{\end{enumerate}}

\newcommand{\bi}{\begin{itemize}}
\newcommand{\ei}{\end{itemize}}

\newcommand{\beq}{\begin{equation}}
\newcommand{\eeq}{\end{equation}}

\newcommand{\bc}{\begin{center}}
\newcommand{\ec}{\end{center}}

\begin{document}
\begin{center}
\LARGE{\bf Describing groups using first-order language}\\[10pt]
\large{Yuki Maehara}\\
\large{Supervisor: Andr\'e Nies}
\end{center}
\vspace{20pt}

\section{Introduction}
How can large groups be described efficiently?
Of course one can always use natural language, or give presentations to be more rigorous, but how about using formal language?
In this paper, we will investigate two notions concerning such descriptions; \emph{quasi-finite axiomatizability}, concerning infinite groups, and \emph{polylogarithmic compressibility}, concerning classes of finite groups.

An infinite group is said to be \emph{quasi-finitely axiomatizable} if it can be described by a single first-order sentence, together with the information that the group is finitely generated (which is not first-order expressible).
In first-order language, the only parts of a sentence that can contain an infinite amount of information are the quantifiers; $\exists$ and $\forall$.
Since the variables correspond to the elements in the language of groups, we can only ``talk'' about elements, but not, for example, all subgroups of a given group.
We give several examples of groups that can be described in this restricted language, with proofs intended to be understandable even to undergraduate students.

We say a class of finite groups is \emph{polylogarithmically compressible} if each group in the class can be described by a first-order sentence, whose length is order polylogarithmic (i.e.~polynomial in $\mathtt{log}$) to the size of the group.
We need a restriction on the length of the sentence because each finite group can be described by a first-order sentence.
The most standard (and inefficient) way to do so is to describe the whole Cayley table, in which case the length of the sentence has order square of the size of the group.
The examples given in this paper include the class of finite simple groups (excluding a certain family), and the class of finite abelian groups.

\section{Preliminaries} \label{preliminaries}
\subsection{Homomorphisms} \label{homomorphisms}
A surjective homomorphism is called an {\it epimorphism}.

A homomorphism from a structure into itself is called an {\it endomorphism}.
If $G$ is an abelian group, then endomorphisms of $G$ form a ring under usual addition and composition (see \cite[\S 5]{Kar.Mer:79}).
In particular, if $x$ (not necessarily in $G$) is such that conjugation by $x$ is an automorphism of $G$, and if $P(X)=\sum_{i \in I}\alpha_i X^i$ is a polynomial over $\ZZ$ where $I$ is a finite subset of $\ZZ$, then for any element $g \in G$, we write
\[
g^{P(x)}=\sum_{i \in I} \alpha_i g^{x^i}
\]
if $G$ is written additively, and similarly
\[
g^{P(x)}=\prod_{i \in I} (g^{x^i})^{\alpha_i}
\]
if $G$ is written multiplicatively.

If $G$ is a group and $x \in G$, then the map defined by $g \mapsto x^{-1}gx$ is an automorphism of $G$ called {\it conjugation} by $x$.
We denote $x^{-1}gx$ by $g^x$, or $\mathtt{Conj}(g,x)$.

\subsection{Presentations} \label{presentations}
\emph{Presentations} are a way of describing groups.
A group $G$ has a presentation
\[
P = \langle~X~|~R~\rangle
\]
where $X$ is the set of generators and $R$ is the set of relators written in terms of the generators in $X$, iff
\[
G \cong F(X)/N
\]
where $F(X)$ is the free group generated by $X$ and $N$ is the least normal subgroup containing $R$.
$G$ is said to be \emph{finitely presented} (\emph{f.p.}) if $X,R$ can be chosen to be finite.

\subsection{Metabelian groups} \label{metabelian}
For any groups $G,A,C$, one says $G = A \rtimes C$ ($G$ is a {\it semidirect product} of $A$ and $C$) if
\begin{center}
$AC=G$,  $A \triangleleft G$, and $A\cap C= \{1\}$.
\end{center}
In particular, $G$ is said to be {\it metabelian} if $G = A \rtimes C$ for some abelian groups $A, C$. This is equivalent to saying that $G'$ is abelian.

\begin{lemma}[Nies \cite{Nies:03}] \label{metabelian commutators}
Let $G = A \rtimes C$ where $A$ and $C$ are abelian.
Then, the commutator subgroup of $G$ is $G' = [A, C]$.
Moreover, if $C$ is generated by a single element $d$, then $G' = \{[u,d]~|~u \in A \}$.
In particular, $G'$ coincides with the set of commutators.
\end{lemma}
\begin{proof}
Let $u,v \in A$ and $x,y \in C$.
Then, $[ux,vy] = [u^x, y][x, v^y]$ using the commutator rule ${[ab,c] = [a,c]^b[b,c]}$ (\cite[3.2.(3)]{Kar.Mer:79}) and since $A$ and $C$ are abelian.
Since both $[u^x, y]$ and $[x, v^y]$ are in $[A,C]$, $G' = [A, C]$.

Now suppose $C = \langle d \rangle$.
We use additive notation in $A$.
Then, clearly $S = {\{[u,d] \ | \ u \in A \}}$ is a subset of $G' = [A, \langle d \rangle]$.
Since $[u,d][v,d]=[u+v,d]$ and $[u,d]^{-1}=[-u,d]$, $S$ forms a subgroup of $G$.
Also, since $[u,x^{-1}]=[-u^{x^{-1}},x]$ and ${[u, d^{n+1}]=[u^{d^n}+ \ldots +u^d+u,d]}$ for all $x \in C$, $S$~contains all commutators.
Hence $G'=S$.
\end{proof}

\subsection{Finitely generated groups} \label{f.g.}
One says a group $G$ is {\it finitely generated} ({\it f.g.}) if $G$ is generated by some finite subset.

\begin{lemma} \label{f.g. factor}
Let $G$ be a f.g.\ group and let $N$ be a normal subgroup of $G$.
Then the factor group $G/N$ is also f.g.
\end{lemma}
\begin{proof}
Let $S = \{g_1, \ldots, g_n\}$ be a finite generating set of $G$.
Then, $N$ is generated by $SN = \{g_i N \ | \ 1 \le i \le n \}$, which is clearly finite.
\end{proof}

\begin{lemma}[Kargapolov and Merzljakov \cite{Kar.Mer:79}] \label{f.g. abelian}
Each f.g.\ abelian group is a finite direct sum of cyclic groups.
\end{lemma}

\begin{lemma} \label{prod conj}
If a group $G$ is generated by a subset $X \cup Y$ of $G$, then each element $g \in G$ can be written as a product
\begin{center}
$g = w \cdot \prod \mathtt{Conj}(x_i, y_i)$ \ where $x_i \in X$, \ $w, y_i \in \langle Y \rangle$.
\end{center}
In particular, if $A$ is a normal subgroup of $G$ such that $X \subseteq A$ and $\langle Y \rangle \cap A$ is trivial, then each element $u \in A$ can be written as 
\begin{center}
$u = \prod \mathtt{Conj}(x_i, y_i)$ \ where $x_i \in X$, \ $y_i \in \langle Y \rangle$.
\end{center}
\end{lemma}

\begin{proof}
Let $g \in G$.
Then since $X \cup Y$ generates $G$, it can be written as a product
\[
g = \prod_{j=1}^n v_j u_j
\]
for some positive integer $n$, where $u_j \in \langle X \rangle$, $v_j \in \langle Y \rangle$ for each $j$.
The lemma can be proven by induction on $n$.
\end{proof}

\subsection{Rings} \label{rings}
In this paper, we mean by a \textquotedblleft ring\textquotedblright \ an associative ring with the multiplicative identity $1 \neq 0$.

An ideal of a ring is said to be {\it principal} if it is generated by a single element.
One says a ring $\mathcal{R}$ is {\it principal} if every ideal of $\mathcal{R}$ is principal.

Let $\mathcal{R}$ be a ring.
A non-zero element $u \in \mathcal{R}$ is called a {\it zero-divisor} if there exists a non-zero element $v \in \mathcal{R}$ such that $uv=0$ or $vu=0$.
$\mathcal{R}$ is said to be {\it entire} if it is commutative and does not contain any zero-divisors.

Let $\mathcal{R}$ be a commutative ring and let $S$ be a subset of $\mathcal{R}$ containing $1$, closed under multiplication.
Define a relation $\sim$ on the set ${\{(a,s) \ | \ a \in \mathcal{R}, \ s \in S\}}$ by
\begin{center}
$(a,s) \sim (a',s')$ iff there exists $s_1 \in S$ such that $s_1(s'a-sa')=0$,
\end{center}
then clearly $\sim$ is an equivalence relation.
Let $S^{-1}\mathcal{R}$ be the set of equivalence classes and let $a/s$ denote the equivalence class containing $(a,s)$. 
Then $S^{-1}\mathcal{R}$ forms a ring under the following operations;
\begin{itemize}
\item addition is defined by $a/s + a'/s' = (s'a+sa')/ss'$
\item multiplication is defined by $(a/s) \cdot (a'/s') = aa'/ss'$
\end{itemize}
This ring is called the {\it ring of fractions} of $\mathcal{R}$ by $S$.
For the proofs that these operations are well-defined and $S^{-1}\mathcal{R}$ forms a ring, see \cite[\S 3]{Lang:84}.

\subsection{Modules} \label{modules}
Let $\mathcal{R}$ be a ring.
A {\it module} over $\mathcal{R}$, or an {\it $\mathcal{R}$-module} $M$, written additively, is an abelian group with multiplication by elements in $\mathcal{R}$ defined in such a way that for any $a,b \in \mathcal{R}$ and for any $x,y \in M$,
\begin{itemize}
\item $(a+b)x=ax+bx$ \ and
\item $a(x+y)=ax+ay$.
\end{itemize}
Modules are generalization of abelian groups in the sense that every abelian group is a $\ZZ$-module.

A {\it generating set} $S$ of an $\mathcal{R}$-module $M$ is a subset of $M$ such that every element of $M$ can be written as a sum of terms in the form $a_i s_i$ where $a_i \in \mathcal{R}$, $s_i \in S$.
A module is said to be {\it finitely generated} if it possesses a finite generating set.

One says an $\mathcal{R}$-module $M$ is {\it torsion-free} if for any $a \in \mathcal{R}$, $x \in M$, $ax=0$ implies $a=0$ or $x=0$.

An $\mathcal{R}$-module $M$ is {\it free} if it is isomorphic to $\bigoplus_{i \in I} R_i$ for some finite index $I$, where each $R_i$ is isomorphic to $\mathcal{R}$ seen as a module over itself in the natural way.

\section{Quasi-finitely axiomatizable groups} \label{QFA}

\begin{definition} \label{QFA def}
An infinite f.g.\ group $H$ is \emph{quasi-finitely axiomatizable (QFA)} if there exists a first-order sentence $\psi$ such that $H \models \psi$ and if $G$ is a f.g.\ group and $G \models \psi$, then $G \cong H$.
\end{definition}

The idea of quasi-finite axiomatizability was introduced by Nies in \cite{Nies:03}.
It was originally used to determine the expressiveness of first-order logic in group theory.
Later, it turned out to be interesting even from an algebraic point of view.
For example, Oger and Sabbagh \cite{Oger.Sabbagh:06} showed that if $G$ is a f.g.~nilpotent group, then $G$ is QFA iff each element $z \in Z(G)$ satisfies $z^n \in G'$ for some positive integer $n$.

In each of the proofs below, we give the sentence $\psi$, suppose a f.g.~group $G$ satisfies $\psi$ and show that $G$ must be isomorphic to $H$.
Since we ``do not know'' whether $G \cong H$ holds until the end of the proof, we want to talk about $G$ and $H$ separately.
So we refer to the group $H$ as the {\it standard case}, as opposed to the {\it general case} $G$.

\subsection{Finitely presented groups} \label{BS}
Our first example is Baumslag-Solitar groups, which are finitely presented (see below for presentations).
They are relatively easy to describe in first-order because the whole presentation can be a part of the sentence.
Although most of the QFA groups we give in this paper are described as semidirect products, this is the only case where we can define the action in first-order.

A Baumslag-Solitar group is a group with a presentation of the form
\[
\langle \ a,d \ | \ d^{-1}a^n d= a^m \rangle
\]
for some integers $m,n$.
We show that each Baumslag-Solitar group is QFA for the cases where $m \ge 2$ and $n=1$.
For each $m \ge 2$, define
\[
H_m = \langle~a,d \ | \ d^{-1}ad = a^m \rangle.
\]
Then $H_m$ is the semidirect product of $A = \ZZ[1/m] =  \{z m^{-i} \ |\ z \in \ZZ, i \in \NN \}$ by  $C = \langle d \rangle$, where the action of $d$ on $A$ is given by $d^{-1}ud = um$ for $u \in A$.
 
\begin{theorem}[Nies \cite{Nies:07}] \label{BS thm}
$H_m$ is QFA for each integer $m \ge 2$.
\end{theorem}
\begin{proof}
We prove the theorem by actually giving a sentence $\exists d \ \phi_m(d)$ describing the group.
We list sufficiently many first-order properties (P0)-(P9) of $H_m$ and its element $d$ so that whenever a f.g.\ group $G$ has an element $d$ satisfying the conjunction $\phi_m(d) \equiv {{\rm[(P0)} \wedge \ldots \wedge {\rm(P9)]}}$, we must have $G \cong H_m$. That is, \ $G = A \rtimes C$ where each of $A,C$ is isomorphic to that of the standard case, and $C$ is generated by $d$.

The first two are properties of $d$.
\begin{itemize}
\item[(P0)] $d \neq 1$
\item[(P1)] $\fa g \, [g^i \neq d]$, for each $i$, $1 < i \le m$. 
\end{itemize}

The next five formulas define the subgroups $G', A, C$ and describe some of their properties.
Fix a prime $q$ that does not divide $m$.
\begin{itemize}
\item[(P2)] The commutators form a subgroup (so that $G'$ is definable)
\item[(P3)] $A=\{g \ | \ g^{m-1} \in G'\}$ and $C=C(d)$ are abelian, and $G = A \rtimes C$
\item[(P4)] $|C:C^2|=2$
\item[(P5)] $|A:A^q|=q$
\item[(P6)] The map $u \mapsto u^q$ is 1-1 in $A$.
\end{itemize}
We know (P2) holds in $H_m$ from Lemma \ref{metabelian commutators}.
(P4) can be expressed as \textquotedblleft there is an element which is not the square of any element, and for any three elements $x_1, x_2, x_3$, the formula $\exists y \ [x_i = x_j y^2]$ is satisfied for some $1 \le i < j \le 3$\textquotedblright, and similar for (P5).
 
We show that (P3) actually defines $A$ in the standard case.
If $g \in \ZZ[1/m]$, then $[g,d]=g^{-1}d^{-1}gd=g^{-1}g^{m}=g^{m-1}$, so $g^{m-1} \in H'_m$.
Conversely, since $\ZZ[1/m]$ is closed under taking roots (because  $H_m/\ZZ[1/m]$ is torsion-free), $g \notin \ZZ[1/m]$ implies $g^{m-1} \notin \ZZ[1/m]$.
Since  $H_m/\ZZ[1/m]$ is abelian, $H'_m \le \ZZ[1/m]$ and so $g^{m-1} \notin H'_m$.

The last three describe how $C$ acts on $A$.
\begin{itemize}
\item[(P7)] $\fa u \in A \ [ d^{-1}u d= u^m]$
\item[(P8)] $u^x \in A-\{1,u\}$ for $u \in A- \{1\}$, $x \in C -\{1\}$
\item[(P9)] $u^x \neq u^{-1}$ for $u \in A-\{1\}$, $x \in C$
\end{itemize}
(P8) says that $C-\{1\}$ acts on $A-\{1\}$ without fixed points, and (P9) says that the orbit of $u \in A$ under $C$ does not contain $u^{-1}$, unless $u=1$.

Now let $G$ be a f.g.\ group and suppose $d \in G$ satisfies (P0)-(P9).
First, we show that the order of $d$ is infinite.
If $d^r = 1$ for some $r > 0$, then for each $u \in A$ we have $u = d^{-r}ud^r = u^{mr}$.
So $A$ is a periodic group of some exponent $k \le mr-1$.
If $q$ divides $k$, then there exists an element $v \in A$ of order $q$.
This makes the map $u \mapsto u^q$ not 1-1, contrary to (P6).
If $q$ does not divide $k$, then the map is an automorphism of $A$ and so $A^q=A$, contrary to (P5).

Let $\mathcal{R} = \ZZ[1/m]$ viewed as a ring.
Then $A$ can be seen as an $\mathcal{R}$-module by defining $u(zm^{-i}) = \mathtt{Conj}(u^z,d^{-i}) \ ( = u^{zm^{-i}}$, so well-defined) for $z \in \ZZ$, $i \in \NN$.
Now we show that $A$ is f.g.\ and torsion-free as an $\mathcal{R}$-module.

Since $C \cong G/A$ and $G$ is f.g., $C$ is f.g.\ abelian by Lemma \ref{f.g. factor} and (P3).
So $C$ is a direct sum of cyclic groups by Lemma \ref{f.g. abelian}, and has only one infinite cyclic factor by (P4).
Since $d$ has infinite order, we can choose a generator $c \in C$ of this factor that satisfies $c^s = d$ for some $s \ge 1$.
Then, $C = \langle c \rangle \times F$ where $F = T(C)$ is the torsion subgroup of C.
Since $G = AC$, $G$~has a finite generating set of the form $B \cup \{ c \} \cup F$, where $B \subseteq A$. We may assume $B$ is closed under taking inverse, and under conjugation by elements of the set $F \cup \{ c^i \ | \ 1 \le i < s\}$.
If $u \in A$, then $u$ can be written as a product of the terms $\mathtt{Conj}(b, xc^z)$ where $x \in F$, $z \in \ZZ$, $b \in B$, by Lemma \ref{prod conj}.
Hence by the closure properties of $B$, $u$ is a product of the terms $\mathtt{Conj}(b', d^w)$ where $b' \in B$ and $w \in \ZZ$.
This shows that $A$ is f.g.\ as an $\mathcal{R}$-module.

Suppose $u(zm^{-i}) = \mathtt{Conj}(u^z, d^{-i}) = 1$ for some $u \neq 1$, $i \ge 0$, $z \neq 0$.
Then $u^z = 1$ by (P8), so conjugation by $d$ is an automorphism of the finite subgroup $\langle u \rangle$ by (P7).
Hence some power of $d$ has a fixed point, contrary to (P8).
Therefore $A$ is torsion-free as an $\mathcal{R}$-module.

Since $\mathcal{R}$ is a principal entire ring, $A$ is a free $\mathcal{R}$-module by \cite[Thm.\ XV.2.2]{Lang:84}, so that $A$ as a group is isomorphic to $\bigoplus_{1 \le i \le k} R_i$ for some positive integer $k$, where each $R_i$ is isomorphic to the additive group of $\mathcal{R}$.
But then $|A:A^q| = q^k$, so $k = 1$ by (P5).

Now we show that $F$ is trivial.
For, suppose $x \in F$, then the action of $x$ is an automorphism of $\ZZ [1/m]$ of finite order.
Note $\mathtt{Aut}(\ZZ[1/m])$ is isomorphic to $\ZZ \times \ZZ_2$ where the first factor is generated by the map $u \mapsto um$ and the second by the map $u \mapsto -u$.
So the action of $x$ is either the identity, or inversion.
But we know that $x$ cannot be inversion by (P9), so $F = \{1\}$.

Recall that $s$ is the positive integer satisfying $c^s=d$.
Then $s \le m$ because the automorphism $u \mapsto um$ is not an $i$th power in $\mathtt{Aut}(\ZZ [1/m])$ for any $i > m$.
So $i = 1$ by (P1), meaning $\langle d \rangle = C$.
Hence $G \cong H_m$.
\end{proof}
(P1) was needed in the last part because $d$ might be a proper power of $c$.
For example, if $m=4$, then $\ZZ[1/4]=\ZZ[1/2]$ and the map $u \mapsto 4u$ is clearly the square of the map $u \mapsto 2u$, which is an automorphism of $\ZZ[1/4]$.

\subsection{Non-finitely presented groups} \label{wreath}
As mentioned in the last subsection, the groups in this example are also described as semidirect products, but in this case we cannot define the action explicitly.
Instead, we use a relationship between (definable) subgroups to restrict our possibilities.

The restricted wreath product $\ZZ_p \wr \ZZ$ is the semidirect product $H_p = A \rtimes C$ where $A = \bigoplus_{z \in \ZZ} \ZZ_p^{(z)}$,
$\ZZ_p^{(z)}$ is a copy of $\ZZ_p$, $C = \langle d \rangle$ with $d$ of infinite order,
and $d$ acts on $A$ by shifting, i.e.\ $(\ZZ_p^{(z)})^d = \ZZ_p^{(z+1)}$.
It has a presentation
\begin{equation} \label{presentation 1}
\langle~a,d \ | \ a^p, [v_r, v_s] (r,s \in \ZZ, r<s) \rangle
\end{equation}
where $a$ corresponds to a generator of $\ZZ_p^{(0)}$ and $v_r = a^{d^r}$.

\begin{theorem}[Nies \cite{Nies:03}] \label{wreath thm}
$\ZZ_p \wr \ZZ$ is QFA for each prime $p$.
\end{theorem}

\begin{proof}
The general idea of the proof is the same as that of the previous example.
We express the group as a semidirect product of its subgroups $A,C$, and then show that each of them is isomorphic to that in the standard case.
We also need to make sure that $C$ acts on $A$ correctly in this case, since the action of $d$ cannot be expressed in first-order.

Let $H_p = \ZZ_p \wr \ZZ$.
Then the sentence describing $H_p$ is $\exists a \exists d \ \phi_p(a,d)$ where $\phi_p(a,d) \equiv$ [(P0) $\wedge \ \ldots \ \wedge$ (P6)].
We use additive notation in $A$.

The first formula says that neither of $a,d$ is the identity.
Note both 0 and 1 refer to the identity here, since $a$ is in $A$ where we use additive notation, while $d$ is in $C$ where we use multiplicative notation.
\begin{itemize}
\item[(P0)] $a \neq 0$, $p \cdot a = 0$, $d \neq 1$.
\end{itemize}

The next five define the subgroups $G', A, C$ and describe some of their properties.
Note that $\langle a \rangle$ is first-order definable since it is finite by (P0) above.
\begin{itemize}
\item[(P1)] The commutators form a subgroup (so that $G'$ is definable)
\item[(P2)] $A = G' + \langle a \rangle = G' \oplus \langle a \rangle$ and $C = C(d)$ are abelian, and $G = A \rtimes C$
\item[(P3)] $|C:C^2|=2$
\item[(P4)] $\forall u \in A \ [p \cdot u=0]$
\item[(P5)] No element in $C-\{1\}$ has order $< p$.
\end{itemize}

The last thing we need to say is that $C-\{1\}$ acts on $A-\{0\}$ without fixed points.
\begin{itemize}
\item[(P6)] $u^x \in A-\{0, u\}$ for $u \in A-\{0\}$, $x \in C-\{1\}$.
\end{itemize}

First, we show that ${A = H'_p \oplus \langle a \rangle}$ holds in the standard case.
We know that it has the form ${H'_p = \{[u,d] \ | \ u \in A \}}$ from Lemma~\ref{metabelian commutators}, and so it is a subgroup of $A$.
Consider the group $\tilde{H}_p$ with a presentation
\begin{equation} \label{presentation 2}
\langle~\tilde{a}, \tilde{d} \ | \ {\tilde{a}}^p, [\tilde{a},\tilde{d}]~\rangle.
\end{equation}
Since for each relator in (\ref{presentation 1}), a corresponding relator is in (\ref{presentation 2}), there exists an epimorphism $\Psi: H_p \rightarrow \tilde{H}_p$ mapping $a, d$ to $\tilde{a}, \tilde{d}$ respectively.
As $\mathtt{Ker}(\Psi)$ is properly contained in $A$ and $\tilde{H}_p$ is abelian, $H'_p$ is properly contained in $A$.
Now for each $z \neq 0$, $\ZZ_p^{(z)}$ is generated by $a^{d^z}=a+[a,d^z]$, so $H'_p+\langle a \rangle=A$.
If $a^r \in H'_p$ for some $0<r<p$, then there exists $u \in A$ such that $a^r=[u,d]$ or equivalently $u+a^r=u^d$ i.e.\ $a^r$ shifts $u$, which is impossible.
Hence $H'_p \cap \langle a \rangle$ is trivial and so $A = H'_p \oplus \langle a \rangle$.

Now let $G$ be a f.g.\ group and suppose $a,d \in G$ satisfy (P0)-(P6).
We first prove that $C$ is infinite cyclic.
Since $C$ is f.g.\ abelian of torsion-free rank 1 by (P3), it suffices to show that $C$ is torsion-free by Lemma \ref{f.g. abelian}.

For, suppose $t \in C-\{1\}$ has finite order $r$.
Then every orbit in $A-\{0\}$ under the action of $t$ has size $r$, because if some orbit has size $s<r$, then $t^s \in C-\{1\}$ has a fixed point.
Let $A$ be viewed as a vector space over $\ZZ_p$ and let $U$ be the $t$-invariant subspace of $A$ generated by $a$.
Then $|U| = p^n$ for some $1 \le n \le r$, because $U = \left\{\sum_{0 \le i < r} m_i \cdot a^{t^i} \ | \ m_i \in \ZZ_p \right\}$.
But $|U-\{0\}| \ge p$ by (P5) and so $n > 1$.
Now consider the size of $G' \cap U$, which is also $t$-invariant because $G'$ is normal in~$G$.
Since $a$ is not in $G'$, $|U:G' \cap U| > 1$.
We also know that $|U:G' \cap U| \le |A:G'|=p$ \mbox{from \cite[Exercise 2.4.4]{Kar.Mer:79}}.
Hence the only possible size is $p^{n-1}$ since it must divide $|U|=p^n$.
As every orbit has size $r$ and the orbits partition each $t$-invariant subspace excluding the identity, $r$ divides $p^n-1$ and $p^{n-1}-1$.
But $(p^n-1)-p(p^{n-1}-1)=p-1$, so $r$ also divides $p-1$.
In particular, $r \le p-1$, contrary to (P5).

Choose a generator $c$ of $C$ and let $\mathcal{R}$ be the ring of fractions of $\ZZ_p[c]$ by the multiplicative subset $\{c^n \ | \ n \ge 0\}$.
Then $\mathcal{R}$ is a principal entire ring because the polynomial ring $\ZZ_p[c]$ is principal entire (see \cite[Section II.3 and Exercise 4]{Lang:84}).

Now, $A$ can be seen as an $\mathcal{R}$-module by defining $u \cdot P =  \sum_{i=r}^s\alpha_i u^{c^i}$ for $u \in A$, $P = \sum_{i=r}^s\alpha_ic^i \in \mathcal{R}$.
We show that $A$ is f.g.\ and torsion-free as an $\mathcal{R}$-module.

Let $B=\{b_1,\ldots,b_m\}$ be a finite generating set of $G$.
Then, since each $b_i \in B$ can be written in the form $b_i = u_i c^{z_i}$ where $u_i \in A$, $z_i \in \ZZ$, the set $S \cup \{c\}$ also generates $G$ where $S = \{u_1,\ldots,u_m\}$.
Hence every element $u$ in $A$ can be written as a sum of the terms $\mathtt{Conj}(u_j, c^{z_j})$ where $u_j \in S$, $z_j \in \ZZ$ by Lemma \ref{prod conj}, meaning $A$ is f.g.\ as an $\mathcal{R}$-module.

Suppose $u \cdot P=0$ for some $u \in A - \{0\}$, $P = \sum_{i=r}^s\alpha_ic^i \in \mathcal{R} - \{0\}$.
Then $P$ must consist of more than one term, for if $\alpha u^{c^z} = 0$ for some $\alpha \neq 0$, then $u^{c^z} = 0$ by~(P4), contrary to (P6).
We can assume that the leading coefficient of $P$ is $-1$, so that $u^{c^s} = \sum_{i=r}^{s-1} \alpha_i u^{c^i}$.
But then, for each $w \ge s$, $u^{c^w}$ is in the finite subspace of~$A$ generated by ${\{ u^{c^i} \ | \ r \le i \le s-1 \}}$.
\footnote{e.g.
\[
\begin{split}
u^{c^{s+1}}&={\sum_{i=r}^{s-1} \alpha_i u^{c^{i+1}}}\\
&={\alpha_{s-1} u^{c^s} +  \sum_{i=r}^{s-2} \alpha_i u^{c^{i+1}}}\\
&={\sum_{i=r}^{s-1} [\alpha_{s-1} \alpha_i + \alpha_{i-1}]u^{c^i}}
\end{split}
\]
for the same $\alpha_i$ as above except $\alpha_{r-1}=0$.}
Hence the action of some power of $c$ has a fixed point, contrary to (P6).
This shows that $A$ is torsion-free as an $\mathcal{R}$-module.

Recall $\mathcal{R}$ is a principal entire ring. Since $A$ f.g.\ and torsion-free as an $\mathcal{R}$-module, it is a free module by \cite[Thm.\ XV.2.2]{Lang:84} so that $A$ as a group is isomorphic to $\bigoplus_{1 \le i \le k} R_i$ for some positive integer $k$, where each $R_i$ is isomorphic to the additive group of $\mathcal{R}$.
Observe that $R_i \rtimes C \cong H_p$ for each $i$, where the action of $c$ on $R_i$ is defined by $P \mapsto P \cdot c$, and so $|R_i:[R_i,C]| = p$.
If $k >1$, then $|A:G'| = \left|\bigoplus_i R_i:\left[\bigoplus_i R_i,C\right]\right| > p$, contrary to (P2).

The last thing we need to show is that the action of $c$ on $A$ is correct.
To avoid confusion, here we denote by $d_H$, $A_H$ one of the generators and the normal subgroup of $H_p$ respectively.
Since $c$ has infinite order and each power of $c$ (except the identity) acts without fixed points, the action of $c$ on $A$ is equivalent to the action of $d_H^m$ on $A_H$ for some $m \ge 1$.
But if $m > 1$, then $A \nsubseteq G' \oplus \langle a \rangle$, contrary to~(P2).
\end{proof}

\subsection{Semidirect products of f.g.~groups} \label{Oger}

In \cite{Oger:06}, Oger gave examples of QFA groups, which are semidirect products of $\ZZ[u]$ and infinite cyclic $\langle u \rangle$ where $u$ is a complex number satisfying certain conditions.
Since both $\ZZ[u], \langle u \rangle$ are f.g.~abelian, we can talk about the rank of $\ZZ[u]$ as a free abelian group, making the proof fairly different from the previous examples.

Let $\mathcal{R}$ be a commutative ring.
An element $\alpha$ of $\mathcal{R}$ is said to be {\it integral} over $\mathcal{R}$ if there exists a monic (i.e.\ the leading coefficient is 1) polynomial $P$ over $\mathcal{R}$ such that $P(\alpha)=0$.
Let $\mathcal{S}$ be a commutative ring containing $\mathcal{R}$ as a subring.
Then, the elements of $\mathcal{S}$ integral over $\mathcal{R}$ form a subring of $\mathcal{S}$.
This ring is called the {\it integral closure} of $\mathcal{R}$ in $\mathcal{S}$ (see \cite[IX, \S1]{Lang:84}).

\begin{theorem}[Oger \cite{Oger:06}] \label{Oger thm}
Let $u$ be a complex number such that
\begin{itemize}
\item $\ZZ[u]$ is the integral closure of $\ZZ$ in $\QQ[u]$
\item the multiplicative group $(\ZZ[u]^*, \times)$ is infinite and generated by $u$ and $-1$.
\end{itemize}
Then there exists a first-order sentence $\psi$ which characterizes, among f.g.\ groups, those which are isomorphic to semidirect products $A \rtimes \langle u \rangle$, where $A$ is a non-zero ideal of $\ZZ[u]$, and the action of $u$ on $A$ is defined by $x \mapsto xu$.
\end{theorem}

\begin{corollary}[\cite{Oger:06}] \label{Oger col}
If $u$ satisfies the conditions above and $\ZZ[u]$ is principal, then $\ZZ[u] \rtimes \langle u \rangle$ is QFA.
\end{corollary}

\begin{proof}
Let $u$ be a complex number that satisfies all of these conditions.
Let $A$ be a non-zero ideal of $\ZZ[u]$.
Then, since $\ZZ[u]$ is principal, there exists $a \in \ZZ[u]$ such that $A =a \cdot\ZZ[u]$.
If we define a map ${\Phi : \ZZ[u] \rightarrow A}$ by $\Phi(x) = ax$, then $\Phi$ is clearly a group isomorphism ${(\ZZ[u], +) \rightarrow (A, +)}$.
Since $\Phi$ also preserves the action of~$u$ (as $\Phi(xu)=axu=\Phi(x) \cdot u$), $\Phi$ can be extended to an isomorphism $\ZZ[u] \rtimes \langle u \rangle \rightarrow A \rtimes \langle u \rangle$.
\end{proof}

One example of such $u$ was given in \cite{Oger:06}, namely $u=2+\sqrt{3}$.
Clearly ${\ZZ[u]=\ZZ[\sqrt{3}]}$ is the integral closure of $\ZZ$ in $\QQ[\sqrt{3}]$.
One can show that each invertible element $x\in \ZZ[\sqrt{3}]$ has the form $x=\pm (2+\sqrt{3})^n$ for some integer $n$ by considering the sequence $\{x_k\}$ defined by $x_0=x$, $x_{k+1}=x_k \cdot (2+\sqrt{3})^{-1}$.
Since the (norm) function $N:\ZZ[\sqrt{3}] \rightarrow \NN$ defined by $N(a+b\sqrt{3})=|a^2-3b^2|$ satisfies the conditions
\begin{itemize}
\item if $y_1, y_2 \neq 0$ then $N(y_1) \le N(y_1 \cdot y_2)$
\item if $y_2 \neq 0$ then there exist $q,r \in \ZZ[\sqrt{3}]$ such that $y_1=q \cdot y_2+r$ and ${f(r) < f(y_2)}$
\end{itemize}
for any $y_1, y_2 \in \ZZ[\sqrt{3}]$, the ring $\ZZ[\sqrt{3}]$ is an Euclidean domain and so is principal (see \cite[5.5]{Ribenboim:01}).

\begin{proof}[Proof of Theorem \ref{Oger thm}]
The first-order sentence describing the semidirect products is $\exists y \exists z \ \phi(y,z)$ where $\phi(y,z) \equiv$ [(P0) $\wedge \ \ldots \ \wedge$ (P6)].
Let $P$ be the minimal polynomial of $u$ over $\ZZ$, and let $n = \mathtt{deg}(P)$.
We use additive notation in $A$.

First, we state that $y,z$ are non-identity elements.
Note both $0,1$ refer to the identity element.
\begin{itemize}
\item[(P0)] $y \neq 0$, $z \neq 1$
\end{itemize}

Next, we define $A,C$ and describe some of their properties.
\begin{itemize}
\item[(P1)] $A = C(y)$ and $C = C(z)$ are abelian, and $G = A \rtimes C$
\item[(P2)] $|A:2A|=2^n$
\item[(P3)] $|C:C^2|=2$
\item[(P4)] $x^k \neq 1$ for $x \in C-\{1\}$, $1 \le k \le n+1$
\end{itemize}

The rest is the following.
\begin{itemize}
\item[(P5)] $\mathtt{Conj}(w,x) \neq w$ for $w \in A-\{0\}$, $x \in C-\{1\}$
\item[(P6)] $P(f)=0$ for the automorphism $f$ of $A$ defined by $w \mapsto w^z$
\end{itemize}
(P6) is equivalent to saying $P(z)=0$, but we need to express it this way because the group operation (which is multiplication when considering $C$) is the only operation we are allowed to use.

Let $G$ be a f.g.\ model of $\psi$.
First, we show that $z$ has infinite order.
For, suppose $z^t = 1$ for some positive integer $t>1$.
Then $f$ is a root of the polynomial $X^t-1$ and so $P$ divides $X^t-1$ since $P$ is also a minimal polynomial of $f$.
But then $u^t-1=0$, contrary to the fact that $u,-1$ generate the infinite multiplicative group $\ZZ[u]^*$.
Hence $z$ has infinite order, in particular, $|C:\langle z \rangle|$ is finite as $C$ is f.g.\ abelian of torsion-free rank 1 by (P1), (P3).

Let $w_1, \ldots, w_r \in A$, $x_1, \ldots, x_r \in C$ such that $\{ w_1 x_1, \ldots, w_r x_r \}$ generates $G$, and let $z_1, \ldots, z_s \in C$ such that $C = z_1 \langle z \rangle \cup \ldots \cup z_s \langle z \rangle$ i.e.\ $\{z_1, \ldots, z_s\}$ contains at least one representative from each coset of $\langle z \rangle$ in $C$.
Then by Lemma \ref{prod conj},
\[
\begin{split}
A&=\langle \{ \mathtt{Conj}(w_i,x) \ | \ 1 \le i \le r, \ x \in C \} \rangle\\
&=\langle \{ \mathtt{Conj}(\mathtt{Conj}(w_i,z_j),z^k) \ | \ 1 \le i \le r, \ 1 \le j \le s, \ k \in \ZZ \} \rangle
\end{split}
\]
because each $x \in C$ can be written in the form $x=z_j \cdot z^k$ for some \mbox{$1 \le j \le s$, $k \in \ZZ$.}
Since $\langle \{ \mathtt{Conj}(w,z^k) \ | \ k \in \ZZ \} \rangle$ is f.g.\ for each $w \in C$ by (P6), this means that $A$ is f.g.

Now we show that $A$ is torsion-free.
For, suppose $w \in A-\{0\}$ is a torsion element.
Then ${\{f^k(w) \ | \ k \in \ZZ \}}$ is contained in the torsion subgroup of $A$, which is finite since $A$ is f.g.\ abelian (see \cite[Exercise 8.1.5]{Kar.Mer:79}).
Hence there exist $k_1,k_2 \in \ZZ$ with $k_1 < k_2$ such that $f^{k_1}(w) = f^{k_2}(w)$.
But this means that $z^{k_2-k_1} \in C-\{1\}$ fixes $f^{k_1}(w)=\mathtt{Conj}(w,z^{k_1}) \in A-\{0\}$, contrary to (P5).

Since $A$ is f.g.\ torsion-free, it is free abelian of rank $n$ by (P2).
Also, the subgroup $A_{(y)} = \langle \{ f^k(y) \ | \ 0 \le k \le n-1 \} \rangle$ of $A$ has rank $n$ by (P6) and the minimality of~$P$.
Hence the action of $z$ on $A_{(y)}$, which has finite index in $A$, is equivalent to the action of $u$ on a non-zero ideal of $\ZZ[u]$, meaning that the action of $z$ on $A$ is also equivalent.

Now we show that $C$ is torsion-free.
Otherwise, there exists ${x \in C-\{1\}}$ of prime order $p \ge n+2$ by (P4).
But then $1 = y^{x^p-1} = (y^{x^{p-1} + \ldots + x + 1})^{x-1}$, or equivalently, $Y^x = Y$ where ${Y = y^{x^{p-1} + \ldots + x + 1}}$, and so $Y=1$ by (P2).
Since  $A$ is torsion-free (in particular, $y$ has infinite order) and ${X^{p-1}+ \ldots + X+1}$ is irreducible \mbox{(see \cite[Exercise IV.5.6]{Lang:05})}, ${y^{x^{p-1} + \ldots + x + 1}=1}$ means that the set $\{y^{x^k} \ | \ 0 \le k \le p-2\}$ generates a free abelian group of rank $p-1 \ge n+1$, which is a subgroup of~$A$.
But $A$ is free abelian of rank $n$, contradiction.

We know $C$ is f.g.\ torsion-free abelian of rank 1, or equivalently, infinite cyclic.
Choose a generator $c$ of $C$.
Then there exists $k \in \ZZ$ such that $c^k = z$.
Define an automorphism $g$ of $A$ by $w \mapsto w^c$ and let $Q$ be the minimal polynomial of $g$ over $\ZZ$.
We show that $\mathtt{deg}(Q)=n$.
Because $g$ is an automorphism of a free abelian group of rank~$n$, $\mathtt{deg}(Q)\le n$.
Also, since ${\langle\{g^{kr}(y) \ | \ 0 \le r \le n-1\}\rangle} = {\langle\{ f^r(y) \ | \ 0 \le r \le n-1 \}\rangle}$ has rank $n$, $\mathtt{deg}(Q) \ge n$.

Choose a root $v \in \mathbb{C}$ of $Q$ and an ideal $I$ of the integral closure of $\ZZ$ in $\QQ[v]$ so that the action of $c$ on $A$ is equivalent to the action of $v$ on $I$.
Because $g^k = f$, we can assume $v^k = u$ and so $\QQ[u] \subseteq \QQ[v]$.
But since both fields have dimension~$n$ over $\QQ$, $\QQ[u] = \QQ[v]$.
Now $v$ belongs to $\ZZ[v]=\ZZ[u]=\langle u,-1 \rangle \cup \{0\}$, so $k = \pm 1$ and $C= \langle c \rangle= \langle z \rangle$. 
\end{proof}

\subsection{Nilpotent groups} \label{UT}
Our last example is a nilpotent group.
We give the definition of (class 2) nilpotency later, and for now we only mention that it has a non-trivial center.
This fact stops us from describing it as a semidirect product, because we do not have the main weapon ``no action has a fixed point'' any more.

Let $U$ be the discrete Heisenberg group $UT_3(\ZZ)$, the group of upper unitriangular matrices (i.e.~the entries on the main diagonal are all $1$ and the entries below the diagonal are all $0$) over $\ZZ$.
Then $U$ is a nilpotent group of \mbox{class 2}.
That is, the $U'$ is contained in the center $Z$.
In fact, by \cite[Exercise 16.1.3]{Kar.Mer:79}, $U$ is isomorphic to the free class 2 nilpotent group $F$ with two generators.
Let $t_{mn}(k)$ denote the 3-by-3 matrix with $1$ in its diagonal entries, $k$ in the $m$-th row  $n$-th column entry and $0$ everywhere else.
Then the generators of $F$ correspond to $a = t_{23}(1)$ and $b = t_{12}(1)$.

The following is a well-known fact about nilpotent groups.
\begin{lemma} \label{f.g. nilpotent}
If $G$ is a nilpotent group such that $G/G'$ is f.g., then every subgroup of $G$ is f.g.
In particular, every subgroup of a f.g.\ nilpotent group is f.g.
\end{lemma}
\begin{proof}
See Robinson \cite[3.1.6, 5.2.17]{Robinson:82} for the proof of the first part.
The second part follows because every factor group of a f.g.~group is f.g.~(Lemma \ref{f.g. factor}).
\end{proof}

QFAness of $U$ can be shown using Oger and Sabbagh's criterion (\cite[Thm.10]{Oger.Sabbagh:06}), but here we give a sentence describing $U$ to make it easier to see how $U$ can be characterized in first-order.
The following facts will be used in the proof of \mbox{Theorem \ref{UT thm}}.

\begin{lemma} \label{nilpotent commutators}
Let $G$ be a nilpotent group of class 2 and let $x,y \in G$.
Then, $[x^{m_1}y^{n_1},x^{m_2}y^{n_2}]=[x,y]^{m_1n_2-m_2n_1}$ for any $m_1,m_2,n_1,n_2 \in \ZZ$.
\end{lemma}
\begin{proof}
First, we show that $[x^m,y^n]=[x,y]^{mn}$ for any $m,n \in \ZZ$.
Note
\[
\begin{split}
y^{-1}x&=xx^{-1} \cdot y^{-1}x \cdot yy^{-1}\\
&=x \cdot [x,y] \cdot y^{-1}
\end{split}
\]
(i.e.\ we get $[x,y]$ every time we swap $y^{-1}$ and~$x$).
Since $G$ is class 2 nilpotent, $[x,y] \in G' \subseteq Z(G)$ and so
\[
\begin{split}
[x^m,y^n]&=x^{-m}y^{-n}x^my^n\\
&=x^{-m}y^{-(n-1)}xyx^{m-1}y^n[x,y]\\
&\hspace{50pt}\vdots\\
&=x^{-m}x^my^{-n}y^n[x,y]^{mn}\\
&=[x,y]^{mn}
\end{split}
\]
holds for positive $m,n$.
Also, since
\[
\begin{split}
[x^{-1},y]&=xy^{-1}x^{-1}y \cdot xx^{-1}\\
&=[y,x]^{x^{-1}}\\
&=[y,x]=[x,y]^{-1}
\end{split}
\]
and similarly $[x,y^{-1}]=[x,y]^{-1}$ holds in $G$, $[x^m,y^n]=[x,y]^{mn}$ holds for any $m,n \in \ZZ$.

Now, since $G' \subseteq Z(G)$, the commutator rule ${[ab,c] = [a,c]^b[b,c]}$ (\cite[3.2.(3)]{Kar.Mer:79}) can be reduced to ${[ab,c] = [a,c][b,c]}$, and by taking inverse we also get $[c,ab]=[c,b][c,a]$.
So we have
\[
\begin{split}
[x^{m_1}y^{n_1},x^{m_2}y^{n_2}]&=[x^{m_1},x^{m_2}y^{n_2}][y^{n_1},x^{m_2}y^{n_2}]\\
&=[x^{m_1},y^{n_2}][x^{m_1},x^{m_2}][y^{n_1},y^{n_2}][y^{n_1},x^{m_2}]\\
&=[x,y]^{m_1n_2} \cdot 1 \cdot 1 \cdot [y,x]^{m_2n_1}\\
&=[x,y]^{m_1n_2-m_2n_1}
\end{split}
\]
as required.
\end{proof}

\begin{lemma} \label{UT center}
The center $Z$ of $UT_3(\ZZ)$ is the infinite cyclic group generated by $c = [a,b] = t_{13}(1)$, which coincides with the set of commutators.
\end{lemma}
\begin{proof}
Since \[ \left[ \left(
\begin{array}{@{}ccc@{}}
1 & \alpha_1 & \beta_1  \\
0 & 1 & \gamma_1 \\
0 & 0 & 1
\end{array}
\right), \left(
\begin{array}{@{}ccc@{}}
1 & \alpha_2 & \beta_2  \\
0 & 1 & \gamma_2 \\
0 & 0 & 1
\end{array}
\right) \right] = t_{13}(\alpha_1\gamma_2-\alpha_2\gamma_1), \]
the center is precisely
$Z = \{t_{13}(z) \ | \ z \in \ZZ\} = \langle c \rangle$.

Now by \cite[Exercise 16.1.3]{Kar.Mer:79}, each element $u \in U$ can be written as $u = a^mb^nc^l$ for some $m,n,l \in \ZZ$ and so
\[
\begin{split}
[u,v]&=[a^{m_1}b^{n_1}c^{l_1},a^{m_2}b^{n_2}c^{l_2}]\\
&=[a^{m_1}b^{n_1},a^{m_2}b^{n_2}]\\
&=[a,b]^{m_1n_2-m_2n_1}\\
&=c^{m_1n_2-m_2n_1} \in \langle c \rangle
\end{split}
\]
for any $u=a^{m_1}b^{n_1}c^{l_1},v=a^{m_2}b^{n_2}c^{l_2} \in U$ by Lemma \ref{nilpotent commutators}.
Hence $U' = \langle c \rangle$.
\end{proof}

As a part of the sentence describing $U$, we use the modified version of a formula first introduced by Mal'cev \cite{Malcev:71}.
The formula $\mu(x,y;a,b)$ with parameters $a,b$ defines the ``square'' operation $M_{a,b}$ on the center $Z$ in the sense that $(Z, \circ, M_{a,b}) \cong (\ZZ, +, Q)$ where $Q=\{ (t,t^2) | \ t \in \ZZ\}$.
The formula is
\[
\begin{split}
\mu(x,y;a,b)\equiv\exists u \exists v\{&[u,a]=[v,b]=1 \ \wedge\\
&x=[a,v]=[u,b] \ \wedge\\
&y=[u,v]\}.
\end{split}
\]
This defines the ``square'' because $[a^m,b^n]=[a,b]^{mn}$ holds in $U$ by Lemma \ref{nilpotent commutators}.
\begin{theorem}[Nies \cite{Nies:03}] \label{UT thm}
$UT_3(\ZZ)$ is QFA.
\end{theorem}
\begin{proof}
The sentence $\psi_U$ consists of four formulas;
\begin{itemize}
\item[(P1)] the center $Z$ coincides with the set of commutators
\item[(P2)] $\exists r \exists s \ \gamma(r,s)$ where $\gamma(r,s)$ is as described below
\item[(P3)] $|Z:Z^2|=2$
\item[(P4)] $|B:B^2|=4$ where $B=G/Z$.
\end{itemize}
Roughly speaking, $\gamma(r,s)$ says $Z$ is linearly orderable, using Lagrange's theorem: an integer is non-negative iff it is the sum of four squares of integers.
Formally, $\gamma(r,s)$ is a formula expressing
\begin{itemize}
\item $\mu(x,y;r,s)$ defines a unary operation $M_{r,s}$ on $Z$
\item let $P_{r,s}=\{u\ |\ \exists v_1 \ldots \exists v_4 \ u=M_{r,s}(v_1) \circ \ldots \circ M_{r,s}(v_4)\}$. Then $x \le y \leftrightarrow y-x \in P_{r,s}$ defines a linear order which turns $Z$ into an ordered abelian group with $[r,s]$ being the least positive element.
\end{itemize}

Let $G$ be a f.g.\ model of $\psi_U$. Since $Z$ is linearly orderable by (P2), it is torsion-free.
Now we show that $B$ is also torsion-free.
If $u \in G-Z$, there exists $v \in G$ such that $[u,v] \neq 1$.
Then for each positive integer $n$, $[u^n, v]=[u,v]^n\neq 1$ by Lemma \ref{nilpotent commutators}.
Hence $u^n \notin Z$.

Since $G$ is f.g.\ and (class 2) nilpotent by (P1), $Z$ is f.g.\ by Lemma \ref{f.g. nilpotent}.
Also, since $Z=G'$ by (P1), $B=G/Z=G/G'$ is abelian.
So we know that $Z, B$ are both f.g.\ torsion-free abelian and have rank 1,2 respectively by (P3) and (P4)
i.e.~$Z \cong \ZZ$, $B \cong \ZZ \oplus \ZZ$.

Now we show that $G$ is generated by two elements.
Let $c,d \in G$ such that the cosets $Zc,Zd$ generate $B$, and let $g,h \in G$ such that the commutator $[g,h]$ generates~$Z$.
Then, there exist $x,y,z,w \in \ZZ$ and $u,v \in Z$ such that $g=uc^xd^y$ and $h=vc^zd^w$.
Hence $[g,h]=[c^xd^y,c^zd^w]=[c,d]^{xw-yz}$ by Lemma \ref{nilpotent commutators}.
But also $[g,h]^r=[c,d]$ for some $r \in \ZZ$ because $Z$ is generated by $[g,h]$.
Since $Z$ is torsion-free, it follows that $xw-yz=r=\pm1$.
Thus $[c,d]$ also generates $Z$ and so the two elements $c,d$ generate $G$.

Because $U$ is the free class 2 nilpotent group of rank 2, there exists an epimorphism $h:U \rightarrow G$ mapping $a,b$ to $c,d$ respectively.
If $h$ is not $1-1$, then $\mathtt{Ker}(h)$ is non-trivial and so it must intersect $Z(U)$ non-trivially as $U$ is nilpotent, \mbox{by \cite[Thm.\ 16.2.3]{Lang:84}}.
But this is impossible, because $h([a,b])=[c,d]$ and so $h$ induces an isomorphism $Z(U) \rightarrow Z(G)$.
Hence $h$ is $1-1$, or equivalently, $h$ is itself an isomorphism.
\end{proof}

\section{Polylogarithmic compressibility} \label{PLC}
As an analogue of quasi-finite axiomatizability, we define \emph{polylogarithmic compressibility} below as a property of a class of finite groups, that the groups can be described by ``short'' first-order sentences in the sense described below.
It makes sense to define it as a property of a class of groups rather than a single group, because the length of the sentence is always constant (and so cannot be compared to the size of the group) if we have only one group.

We define the length $|\psi|$ of a first-order formula $\psi$ to be the number of symbols used in $\psi$.
We assume we have infinitely many variables and so each variable is counted as one symbol.
It usually reduces the length of each (sufficiently long) formula by the factor of $O(\mathtt{log} \ n)$ where $n$ is the number of variables used in the sentence.
This is because, if we have only finitely many variables, then (when $n$ is sufficiently large) the variables in the sentence require extra indices, which have length $O(\mathtt{log} \ n)$.
It can be avoided in some cases by repeating the same variables. e.g.\ the sentence
\[
\forall x_1 \forall x_2 [x_1,x_2]=1 \rightarrow \exists x_3 \exists x_4 [x_3,x_4]=1
\]
is equivalent to
\[
\forall x \forall y [x,y]=1 \rightarrow \exists x \exists y [x,y]=1.
\]

\begin{definition} \label{PLC def}
A class $\mathcal{C}$ of finite groups is  \emph{polylogarithmically compressible (PLC)} if for any $H \in \mathcal{C}$, there exists a first-order sentence $\psi_H$ such that $H \models \psi_H$, $|\psi_H| = O(\mathtt{log}^k|H|)$ for some fixed $k$, and if $G \models \psi_H$ then $G \cong H$.
In particular, we say $\mathcal{C}$ is \emph{logarithmically-compressible (LC)} if $k = 1$.
\end{definition}

Since we allow the polynomial change in the length, PLCness is independent of the particular way we define first-order language.
For example, it does not matter whether we use parentheses or Polish notation (which allows us to write parenthesis-free formulas without ambiguity).

Here we give an example of an LC class to illustrate the definition, namely the cyclic groups of order $2^n$.
The sentence describing $\ZZ_{2^n}$, written additively, consists of three formulass; $\psi \equiv \forall x[\psi_1 \wedge \psi_2 \wedge \psi_3]$ where
\[
\begin{split}
\psi_1(x) &\equiv \forall y [2y \neq x] \ \vee \ \exists z \exists w [(2z=x) \wedge (2w=x) \wedge \forall t [(2t=x) \rightarrow (t=z \vee t=w)]]\\
\psi_2(x) &\equiv \neg \exists x_2 \ldots \exists x_{n+1}  \left[2x=x_2 \wedge \bigwedge_{2 \leq i < n+1} 2x_i=x_{i+1}\wedge x_{n+1} \neq 0\right]\\
\psi_3(x) &\equiv \exists x_1 \ldots \exists x_n \left[\bigwedge_{1 \leq i < n} 2x_i=x_{i+1}\wedge x_n \neq 0 \right].
\end{split}
\]
Note that each part has length $O(n)$.

The first formula $\psi_1$ says that for each element $x$ of the group, either no element $y$ satisfies $2y=x$, or there are exactly 2 such $y$.
This is true in $\ZZ_{2^n}$ because, if $x$ is odd then no $y$ satisfies $2y=x$, and if $2m=x$ for some $m$ in $\ZZ$ then precisely $y_1=m$ and $y_2=m+2^{n-1}$ satisfy the equation in $\ZZ_{2^n}$.

The next formula says $2^n x=0$ for any element $x$ (i.e.\ every element has order $2^i$ where $i \le n$), and the last formula says there exists an element $x_1$ such that $2^{n-1} x_1 \neq 0$.
Clearly both of them hold in $\ZZ_{2^n}$.

Now let $G$ be a group written additively such that $G \models \psi$.
Then since $0 \in G$ and $0+0=0$, there exists exactly one element of order $2^1$ from $\psi_1$.
Similarly, it can be shown that $G$ has at most $2^{i-1}$ elements of order $2^i$ for each $i$.
Since every element of $G$ has order $2^i$ for some $i \le n$ from $\psi_2$, the maximum number of elements $G$ can have is $1+\sum_{1 \le i \le n} 2^{i-1}=2^n$.
But there exists an element of order $2^n$ from~$\psi_2,\psi_3$ and so the cyclic subgroup generated by this element must coincide with the whole group $G$.
In other words, $G \cong \ZZ_{2^n}$.

In this section, we give more examples of PLC and LC classes.
The proofs follow the scheme described below except for the last example:
\begin{itemize}
\item[(i)] We give a presentation for the group $H$ so that if $G \models \psi_H$, then $G$ contains a subgroup $\tilde{G}$ isomorphic to some factor of $H$.
\item[(ii)] We express that the generators of $\tilde{G}$ generate the whole group $G$.
\item[(iii)] We express that $\tilde{G} \cong H$.
\end{itemize}
The following lemmas are used repeatedly.

\begin{lemma} \label{presentation}
Given a finite presentation for a group $H$ with generators $a_1,\ldots,a_m$, there exists a first-order formula $\zeta(x_1,\ldots,x_m)$ such that $H \models \zeta(a_1,\ldots,a_m)$, and if $G \models \zeta(b_1,\ldots,b_m)$ for some group $G$ and its elements $b_1,\ldots,b_m$, then the subgroup $\langle b_1,\ldots,b_m \rangle$ of $G$ is isomorphic to $H/N$ for some normal subgroup $N$ of~$H$.
\end{lemma}

Of course, this lemma is used in the part (i) of the scheme.
The length of the formula $\zeta$ depends on the length of the relators.

\begin{proof}
Let $P=\langle~a_1,\ldots,a_m~|~t_1,\ldots,t_n~\rangle$ be a presentation for $H$.
Note that each relator $t_i$ is first-order definable with parameters $a_1,\ldots,a_m$ since it is a product of the gerenators and their inverses \mbox{(i.e.~$t_i=t_i(a_1,\ldots,a_m)$)}.
Then the formula is
\[
\zeta(x_1,\ldots,x_m) \equiv \bigwedge_{1 \le i \le n} t_i(x_1,\ldots,x_m)=1.
\]
If $G \models \zeta(b_1,\ldots,b_m)$ for some group $G$ and its elements $b_1,\ldots,b_m$, then the subgroup $\tilde{G}=\langle b_1,\ldots,b_m \rangle$ of $G$ has a presentation
\[\langle~x_1,\ldots,x_m~|~t_1,\ldots,t_n, u_1,\ldots~\rangle\]
where each $x_j$ corresponds to $b_j$, and $t_i=t_i(x_1,\ldots,x_m)$, $u_k=u_k(x_1,\ldots,x_m)$ for each $i,k$. Hence $\tilde{G} \cong H/N$ where $N$ is the normal subgroup of $H$ generated \mbox{by $\{u_k(a_1,\ldots,a_m)~|~1 \le k\}$.}
In particular, if $N$ is trivial then $\tilde{G} \cong H$.
\end{proof}

\begin{lemma} \label{repeated squaring}
For each positive integer $n$, there exists a first-order formula $\theta_n(x,y)$ of length $O(\mathtt{log}\ n)$ such that $G \models \theta_n(x,y)$ iff $x^n=y$ in the group $G$.
\end{lemma}

The method used here is called \emph{repeated squaring}.
The formulas $\psi_2, \psi_3$ in the example above are also using this technique.

\begin{proof}
Let $n=\alpha_1\ldots\alpha_k$ written in binary where $k=\lfloor \mathtt{log_2}\ n \rfloor$.
Then the formula $\theta_n$ is
\[
\theta_n(x,y) \equiv \exists y_1 \ldots \exists y_k \left[y_1=x \ \wedge \ y_k=y \ \wedge \ \bigwedge_{1 \le i < k} y_{i+1}=y_i \cdot y_i \cdot x^{\alpha_{i+1}}\right]
\]
where $x^{\alpha_i}=x$ if ${\alpha_i}=1$ and $x^{\alpha_i}=1_G$ if ${\alpha_i}=0$. Clearly $\theta_n$ has length $O(\mathtt{log}\ n)$.

Now we show that the formula is correct, by induction on $k$.
If $k=1$, then the only possibility is $n=1$ and correctness is obvious because the formula is reduced to ${\theta_1(x,y) \equiv \exists y_1 [x=y_1=y]}$.
Suppose $\theta_n(x,y)$ is correct for all $n < 2^{k}$ for some~$k$.
Let $N \in \NN$ such that $2^{k} \le N < 2^{k+1}$ and let $N=\beta_1 \ldots \beta_k$ written in binary.
Then,
\[
\begin{split}
\theta_N(x,y)&\equiv \exists y_1 \ldots \exists y_{k} \left[\bigwedge_{1 \le i < k} y_{i+1}=y_i \cdot y_i \cdot x^{\alpha_{i+1}} \ \wedge \ y_1=x \ \wedge \ y_{k}=y\right]\\
&\equiv \exists y_k \left[\theta_{\tilde{N}}(x,y_{k-1}) \ \wedge \ y_k=y_{k-1} \cdot y_{k-1} \cdot x^{\beta_k} \ \wedge \ y_k=y \right]
\end{split}
\]
where $\tilde{N}=\beta_1 \ldots \beta_{k-1}$.
If $\theta_N(x,y)$ holds in $G$ with witnesses $y_1,\ldots,y_k$, then we have ${y_{k-1}=x^{\tilde{N}}}$ by the inductive hypothesis because $\tilde{N} < 2^k$.
Since $N = 2\tilde{N}+{\beta_k}$, it follows that $y_k=y_{k-1} \cdot y_{k-1} \cdot x^{\beta_k}=x^{2\tilde{N}+{\beta_k}}=x^N$, as required.
\end{proof}

\begin{lemma} \label{finite product}
Given a generating set $S$ of a finite group $G$, every element of $G$ can be written as a product of at most $|G|$ generators in $S$.
\end{lemma}

This lemma, combined with the next one, is used in the part (ii) of the scheme.
The basic idea of the proof is the pigeonhole principle.

\begin{proof}
Let $S=\{s_1,\ldots,s_n\}$ be a generating set of $G$.
Then, each element $g \in G$ can be written as a product
\[
g=\prod_{1 \le i \le m}t_i
\]
for some $m$ where $t_i \in S$ for each $i$.
If $m > |G|$, then there exist $j,k \in \NN$ with $j < k \le m$ and
\[
\prod_{1 \le i \le j}t_i = \prod_{1 \le i \le k}t_i
\]
and so $g$ can also be written as
\[
g=\left(\prod_{1 \le i \le j}t_i\right) \cdot \left(\prod_{k < i \le m}t_i\right)
\]
which is a product of $m-(k-j)$ generators.

We can repeat the same procedure until $g$ is written as a product of no more than $|G|$ generators.
\end{proof}

\begin{lemma} \label{generation}
Let $G$ be a finite group.
Then for each positive integer $n$, there exists a first-order formula $\pi_n(g;x_1,\ldots,x_n)$ with parameters $x_1,\ldots,x_n$ of length $O(n+\mathtt{log}|G|)$ such that ${G \models \pi_n(g;x_1,\ldots,x_n)}$ iff ${g \in \langle x_1,\ldots,x_n \rangle}$.
In other words, $\pi_n$ defines the subgroup ${\langle x_1,\ldots,x_n \rangle}$ of $G$.
\end{lemma}

As mentioned above, this lemma is usually used in the part (ii) of the scheme.
A modified version of the formula is also used in the proof of Theorem \ref{abelian thm} to define a subset of the group consisting of some powers of a certain element.

\begin{proof}
We use a device that originated in computational complexity to show that the set of true quantified boolean formulas is PSPACE complete \cite[Thm 8.9]{Sipser:97}. 
We define the formulas $\delta_i(g;x_1,\ldots,x_n)$ with parameters $x_1,\ldots,x_n$ for each $i \in \NN$ inductively.
For $i=0$,
\[
\delta_0(g;x_1,\ldots,x_n) \equiv \bigvee_{1 \le j \le n}g=x_j \ \vee \ g=1
\]
and for $i>0$,
\[
\begin{split}
\delta_i(g;x_1,\ldots,x_n) \equiv \exists u_i \exists v_i [&g = u_i v_i \ \wedge\\
&\forall w_i [(w_i = u_i \vee w_i = v_i) \rightarrow \delta_{i-1}(w_i;x_1,\ldots,x_n)]].
\end{split}
\]
Note $\delta_i$ has length $O(n+i)$, and $G \models \delta_i(g;x_1,\ldots,x_n)$ iff $g$ can be written as a product of at most $2^i$ $x$'s.

Now let $\tilde{G}=\langle x_1,\ldots,x_n \rangle$.
Then by Lemma \ref{finite product}, each $g \in \tilde{G}$ can be written as a product of at most $|\tilde{G}|$ generators of $\tilde{G}$ (i.e.\ $x_1,\ldots,x_n$).
Hence by defining $\pi_n(g;x_1,\ldots,x_n) \equiv \delta_k(g;x_1,\ldots,x_n)$ where $k = \lceil \mathtt{log_2}|G| \rceil$ (and so $2^k \ge |G| \ge |\tilde{G}|$), we get the required formula.
\end{proof}

\subsection{Simple groups} \label{simple}
It is known that finite simple groups can be classified into 18 infinite families, with exceptions of 26 so-called sporadic groups.

In \cite{Babai:97}, L.~Babai {\it et al.}\ showed that all finite simple groups have `short' presentations, with possible exception of the three families: the projective special unitary groups $PSU_3(q) = {^2A_2(q)}$ where $q$ is a prime-power, the Suzuki groups $Sz(q)={^2B_2(q)}$ where $q=2^{2e+1}$ for some positive integer $e > 1$, and the Ree groups $R(q)={^2G_2(q)}$ where $q=3^{2e+1}$ for some positive integer $e > 1$.
They defined the length $l(P)$ of a presentation $P$ to be the number of characters required to write all the relations (or equivalently relators) in $P$, where the exponents are written in binary, and proved that each of these groups have a presentation of length $O(\mathtt{log}^2|G|)$ where $|G|$ is the size of the group.
Note that $l(P)$ is the maximum number of generators in $P$ the relations can `talk' about.
Since each generator must appear in the relations at least once (otherwise it will have infinite order), this means the number of generators in $P$ is also $O(\log^2|G|)$.

Among the families they missed, two of them were shown to have short presentations by other people.
One is $PSU_3(q)$, shown by Hulpke and Seress \cite{Hulpke:01}.
Given a finite field $F_{q^2}$ for some prime-power $q$, an order 2 automorphism $\alpha$ of $F_{q^2}$ can be defined by $x \mapsto x^q$ and it can be extended in the natural way to the (multiplicative) groups of matrices over $F_{q^2}$.
The {\it special unitary group} $SU_3(q)$ is
\[
SU_3(q) = \{A \in SL_3(q^2) \ | \ A \omega \overline{A}^T = \omega\}
\]
where $\overline{A}=A^{\alpha}$ and $\omega =
\left(
\begin{array}{@{}ccc@{}}
0 & 0 & 1 \\
0 & 1 & 0 \\
1 & 0 & 0
\end{array}
\right)$, and the {\it projective special unitary group} $PSU_3(q)$ is the factor of $SU_3(q)$ by its center.

The other family shown to have short presentations is the Suzuki groups.
In fact, the presentation was given in the original paper by Suzuki \cite{Suzuki:62}, and that was observed by J.~Thompson (personal communication to W.~Kantor) according to Hulpke and Seress \cite{Hulpke:01}.

We use these results to prove the following theorem.

\begin{theorem}
The class of finite simple groups, excluding the family of the Ree groups $R(q)={^2G_2(q)}$, is PLC.
\end{theorem}
\begin{proof}
Let $H$ be a finite simple group not belonging to the family $R(q)$ and let $P=\langle~a_1,\ldots,a_m~|~t_1,\ldots,t_n~\rangle$ be a presentation for $H$ with $l(P)=O(\mathtt{log}^2|H|)$.
Then the sentence describing $H$ is $\psi \equiv \exists a_1 \ldots \exists a_m [\psi_1 \wedge \psi_2 \wedge \psi_3]$ where formulas $\psi_1,\psi_2,\psi_3$ correspond to (i),(ii),(iii) in the scheme respectively.

First, we show that $\psi_1$, the `presentation' for the group (which corresponds to the formula $\zeta$ in Lemma \ref{presentation}), has length $O(\mathtt{log}^2|H|)$.
It suffices to show that we can write each relator in appropriate length.
If $t=a_{\phi(1)}^{z_1} \ldots a_{\phi(k)}^{z_k}$ is a relator, where each $a_{\phi(i)}$ is a generator in $P$, then the number of characters required to write $t$ is $l(t)=k+\sum_{1 \le i \le k} \lfloor \mathtt{log}_2 z_i \rfloor$.
Since the formula
\[
\tau(a_1,\ldots,a_m) \equiv \exists b_1 \ldots \exists b_k \left[\bigwedge_{1 \le i \le k} \theta_{z_i}(a_{\phi(i)},b_i) \wedge \prod_{1 \le i \le k} b_i=1\right]
\]
expresses $t=1$, and has length $\simeq 5k + 10 \sum_{1 \le i \le k} \lfloor \mathtt{log_2}~z_i \rfloor$ by Lemma \ref{repeated squaring}, we obtain the required result.
I.e.\ the formula
\[
\psi_1(a_1,\ldots,a_m) \equiv \bigwedge_{1 \le j \le n} \tau_j(a_1,\ldots,a_{m})
\]
where each $\tau_j$ corresponds to the relator $t_j$, has length approximately
\[
\sum_{1 \le j \le n} \left[ 5k + 10 \sum_{1 \le i \le k} \lfloor \mathtt{log_2}~z_i \rfloor \right]
\]
where $k$ is dependent on $j$.

The second formula $\psi_2$ expresses that $a_1,\ldots,a_n$ generate $H$, using Lemma \ref{generation};
\[
\psi_2(a_1,\ldots,a_m) \equiv \forall h [\pi_k(h;a_1,\ldots,a_m)]
\]
where $k = \lceil \mathtt{log}|H| \rceil$.
(Recall that $H \models \pi_k(h;a_1,\ldots,a_m)$ iff $h \in \langle a_1,\ldots,a_m \rangle$.)
Clearly $\psi_2$ has length $O(\mathtt{log}|H|)$.

Now let $G$ be a group and let $x_1,\ldots,x_m \in G$ such that $G \models \psi_1 \wedge \psi_2(x_1,\ldots,x_m)$.
Then we know that $G$~is generated by the elements $x_1,\ldots,x_m$ and is isomorphic to some factor group of $H$.
But since $H$ is simple, we must have $G \cong H$ unless $G$~is trivial.
Hence the last formula $\psi_3$ is
\[
\psi_3(a_1,\ldots,a_m) \equiv [a_1 \neq 1]
\]
assuming $a_1$ is not the identity element in $H$.
We can make this assumption safely because if $a_1=1$, then we can get a shorter presentation for $H$ by excluding $a_1$ from $P$.
Clearly $\psi_3$ has constant length and so the length of the whole sentence $\psi$ is $O(\mathtt{log}^2|H|)$.
\end{proof}

\subsection{Symmetric groups} \label{symmetric}
In the previous subsection, it was easy to obtain the formula $\psi_2$ because the groups considered were simple.
Something similar happens for the case of the symmetric groups, because for each $n \ge 5$, the alternating group $A_n$ is the only non-trivial normal subgroup of $S_n$ (see \cite[3.2.3]{Robinson:82}).

In \cite{Bray:11}, Bray {\it et al.}\ found presentations of length $O(\mathtt{log}(n))$ for the symmetric groups $S_n$, one of whose generators corresponds to the $n$-cycle $(1,\ldots,n)$.
They defined the length of a presentation in a slightly different way from Babai {\it et al.}~\cite{Babai:97}, but it does not affect the order of a presentation (they included the number of generators).
Note that their presentation for $S_n$ has length $O(\mathtt{log \ log}|S_n|)$ because $|S_n| = n!$.

\begin{theorem} \label{symmetric thm}
The class of symmetric groups $S_n$ is LC.
\end{theorem}

\begin{proof}
We can assume $n \ge 5$ and so $A_n$ is the only non-trivial normal subgroup of $S_n$.
The sentence is $\psi \equiv \exists \eta \exists \sigma_2 \ldots \exists \sigma_k [\psi_1 \wedge \psi_2 \wedge \psi_3]$ where $k$ is the number of generators in the short presentation for $S_n$, $\eta$ corresponds to the $n$-cycle $(1,\ldots,n)$ and $\sigma_2, \ldots, \sigma_k$ correspond to the rest of the generators.

The constructions of $\psi_1, \psi_2$ are exactly the same as that in the previous theorem, and so $\psi_1,\psi_2$ have length $O(\mathtt{log} \ n)$, $O(\mathtt{log}|S_n|)$ respectively.

Now let $G$ be a group and let $x_1, \ldots, x_k \in G$ such that $G \models \psi_1 \wedge \psi_2 (x_1, \ldots, x_k)$.
Then, since $\{1\}$, $A_n$, $S_n$ are the only normal subgroups of $S_n$, the group $G$ is isomorphic to $S_n$, $\ZZ_2$ or $\{1\}$.
So the formula
\[
\psi_3(\eta, \sigma_2, \ldots, \sigma_k) \equiv [\eta \neq 1 \ \wedge \ \eta^2 \neq 1]
\]
guarantees that if $G \models \psi$ then $G \cong S_n$.
Since $\psi_3$ has constant length, the length of the whole sentence $\psi$ is $O(\mathtt{log}|S_n|)$.
\end{proof}

\subsection{Abelian groups} \label{abelian}
It is known that each finite abelian group is isomorphic to a direct product of (finite) cyclic groups, and $\ZZ_m \oplus \ZZ_n \cong \ZZ_{mn}$ iff $m,n$ are coprime.
Hence each finite abelian group can be written as a unique direct product of cyclic groups of prime-power order, up to permutation of the factors.

\begin{theorem} \label{abelian thm}
The class of finite abelian groups is LC.
\end{theorem}
\begin{proof}
Let $H=\bigoplus_{1 \le i \le n} \ZZ_{q_i}$ where each $q_i$ has the form $q_i = p_i^{z_i}$ for some prime $p_i$ and some positive integer $z_i$.
Then $H$ has a presentation
\[
\langle~a_1,\ldots,a_n~|~a_1^{q_1},\ldots,a_n^{q_n}, [a_j,a_k]~(1 \le j < k \le n)~\rangle
\]
where each $a_i$ corresponds to a generator of $\ZZ_{q_i}$.
We follow the scheme again (i.e.\ the sentence $\psi$, written additively, has the form ${\psi \equiv\exists a_1 \ldots \exists a_n [\psi_1 \wedge \psi_2 \wedge \psi_3]}$), but $\psi_1$ is slightly different here.
Since saying ``every element commutes with each other'' requires a shorter formula than saying ``every commutator commutes with each other'', $\psi_1$ is
\[
\psi_1(a_1,\ldots,a_n) \equiv \forall g \forall h [g,h]=0 \ \wedge \ \bigwedge_{1 \le i \le n} \theta_{q_i}(a_i,0).
\]
(Recall that $H \models \theta_n(x,y)$ iff $x^n=y$ holds in $H$.)
Clearly it has length $O(\mathtt{log}|H|)$.

The second formula $\psi_2$ is exactly the same as that of the previous example;
\[
\psi_2(a_1,\ldots,a_n) \equiv \forall g [\pi_k(g;a_1,\ldots,a_n)]
\]
where $k = \lceil \mathtt{log}|H| \rceil$.
It has length $O(\mathtt{log}|H|)$.

Now let $G$ be a group written additively and let $x_1,\ldots,x_n \in G$ such that ${G \models \psi_1 \wedge \psi_2(x_1,\ldots,x_n)}$, then we know $G$ is abelian and generated by $x_1,\ldots,x_n$.
For $G$ to be isomorphic to $H$, it suffices that $p_i^{z_i-1} \cdot x_i \neq 0$ for each $i$, and $x_1,\ldots,x_n$ form an independent set in the sense that if $\sum_{1 \le i \le n} \alpha_i x_i = 0$ for some integers $\alpha_i$, then $\alpha_i x_i = 0$ for each~$i$.

Recall that each $x_i$ satisfies $q_i x_i = p_i^{z_i} x_i = 0$ where $p_i$ is prime.
We define a relation $\sim$ on $\{1,\ldots,n\}$ by $i \sim j$ iff $p_i=p_j$.
Now clearly $\sim$ is an equivalence relation.
We denote by $[i], \Omega$ the equivalence class containing $i$ and the set of all the equivalence classes respectively.
Then it is easy to see that if $S_{[i]} = \{x_j~|~j \in [i]\}$ is independent for each $[i] \in \Omega$, then the whole group $G$ is independent.
So the last formula $\psi_3$ is
\[
\psi_3(a_1,\ldots,a_n) \equiv \bigwedge_{1 \le i \le n} \neg \left[ \theta_{p_i^{z_i-1}}(a_i,0) \right] \ \wedge \ \bigwedge_{[i] \in \Omega} \xi_{[i]}(a_1,\ldots,a_n)
\]
where each $\xi_{[i]}$ expresses that $S_{[i]}+p_i G = \{x_j + p_i G~|~j \in [i]\}$ is independent.
This formula is correct because $S_i$ is independent iff $G/p_i G \cong \bigoplus_{j \in [i]} \ZZ_{p_i}^{(j)}$ where each $\ZZ_{p_i}^{(j)}$ is a copy of  $\ZZ_{p_i}$ iff $S_{[i]}+p_i G$ is independent.
(Strictly speaking, we need the assumption that $S_{[i]}+p_i G$ does not contain the identity element $p_i G$, which is the first part of $\xi_{[i]}$ below.)
As a part of $\xi_{[i]}$, we use a modified version of the formula~$\pi_n$;
\[
\pi'_i(g;x) = \delta_k(g;x)
\]
where $k = \lceil \mathtt{log_2} \ p_i \rceil$ and $\delta_k$ is as defined in Lemma \ref{generation}.
So $G \models \pi'_i(g;x)$ iff $g=z \cdot x$ for some non-negative integer $z \le r$ where $r$ is the smallest power of $2$ not smaller than $p_i$.
In particular, $G \models \pi'_i(z \cdot x;x)$ for all $0 \le z < p_i$.
Note that $\pi'_i$ has length $O(\mathtt{log} \ p_i)$.
Now we are ready to write~$\xi_{[i]}$;
\[
\begin{split}
\xi_{[i]}(a_1,\ldots,a_n) \equiv \ &\bigwedge_{j \in [i]} \nexists b \left[ \theta_{p_i}(b,a_j) \right] \ \wedge\\
&\forall b_1 \ldots \forall b_{\lambda(i)} \left[ \left( \bigwedge_{j \in [i]} \pi'_i(a_j,b_{\phi(j)}) \wedge \exists c \left[ \theta_{p_i} \left(c,\sum_{j \in [i]} b_{\phi(j)}\right) \right] \right) \right.\rightarrow\\
&\hspace{125pt}\left. \exists c_1, \ldots, \exists c_{\lambda(i)} \bigwedge_{j \in [i]} \theta_{p_i} (c_{\phi(j)},b_{\phi(j)}) \right]
\end{split}
\]
where $\lambda(i)$ is the size of $[i]$ and $\phi$ is a bijection from $[i]$ to $\{1,\ldots,\lambda(i)\}$.
The second part says that if a linear combination $\sum_{j \in [i]} z_j a_j$ is in $p_i G$ for some non-negative integers ${z_j \le r = 2^{\lceil \mathtt{log_2}~p_i \rceil}}$, then $z_j a_j \in p_i G$ for each $j$.

Now consider the length of $\psi_3$.
For each $i$, $\psi_3$ contains:
\begin{itemize}
\item one $\theta_{p_i^{z_i-1}}$, which has length $\simeq 10(z_i-1)\lfloor \mathtt{log_2}~p_i \rfloor$
\item three $\theta_{p_i}$, each of which has length $\simeq 10\lfloor \mathtt{log_2}~p_i \rfloor$
\item one $\pi'_i$, which has length $\simeq 26 \lceil \mathtt{log_2}~p_i \rceil$.
\end{itemize}
Hence the length of the formula $\psi_3$ has order of
\[
\sum_{1 \le i \le n} z_i \ \mathtt{log} \ p_i = \mathtt{log} \left( \prod_{1 \le i \le n} p_i^{z_i} \right) = \mathtt{log}|H|
\]
and so the length of the whole sentence $\psi$ is $O(\mathtt{log}|H|)$.
\end{proof}

\subsection{Upper unitriangular matrix groups} \label{unitriangular}
In Subsection \ref{UT}, we analyzed the structure and some properties of the group $UT_3(\ZZ)$.
Using some of those results, and also some part of the previous theorem, we consider similar finite groups, namely ${UT_3(n)=UT_3(\ZZ_n)}$ where $n$ is any positive integer.
Each $UT_3(n)$ is isomorphic to the free 2-generated class 2 nilpotent group with exponent $n$.
Their freeness can be shown in a similar way to the case of $UT_3(\ZZ)$ (see \cite[Exercise 16.1.3]{Kar.Mer:79}).

\begin{proposition} \label{UT finite thm}
The class of the unitriangular groups $UT_3(n)$ is LC.
\end{proposition}

\begin{proof}
Let ${a=t_{23}(1)}$, ${b=t_{12}(1)}$.
(Recall that $t_{mn}(k)$ denotes the 3-by-3 matrix with $1$ in its diagonal entries, $k$ in the $m$-th row  $n$-th column entry and $0$ everywhere else.)
We begin by analyzing the structure of the group $H=UT_3(n)$.
Since \[ \left[ \left(
\begin{array}{@{}ccc@{}}
1 & \alpha_1 & \beta_1  \\
0 & 1 & \gamma_1 \\
0 & 0 & 1
\end{array}
\right), \left(
\begin{array}{@{}ccc@{}}
1 & \alpha_2 & \beta_2  \\
0 & 1 & \gamma_2 \\
0 & 0 & 1
\end{array}
\right) \right] = t_{13}(\alpha_1\gamma_2-\alpha_2\gamma_1), \]
holds in $H$, the center $Z$ of $H$ is
\[
Z=\{t_{13}(z)~|~z \in \ZZ_n\}=\langle c \rangle
\]
where $c=[a,b]=t_{13}(1)$.
Now $H/Z$ is isomorphic to $\ZZ_n \oplus \ZZ_n$ (generated by $aZ,bZ$) which is abelian, so $H$ is class 2 nilpotent.
Hence each element $h \in H$ can be written as a product of the form $h=xyz$ where $x \in \langle a \rangle$, $y \in \langle b \rangle$, $z \in Z$, and $H'$~coincides with the set of commutators by Lemma \ref{nilpotent commutators}.

Let $\phi$ be the formula
\[
\begin{split}
\phi(h,x,y,z;a,b) \equiv \exists u \exists v \{&\pi'_n(u;a)~\wedge~\pi'_n(v;b)~\wedge~\forall w~[z,w]=1~\wedge\\
&h=uvz~\wedge~x=[u,b]~\wedge~y=[a,v]\}
\end{split}
\]
with parameters $a,b$ where $\pi'_n(r;s)=\delta_k(r;s)$, $k=\lceil \mathtt{log_2}~n \rceil$
(i.e.~$H \models \pi'_n(r;s)$ iff $s=r^z$ for some $z \in \ZZ_n$).
Then it defines a bijection $\Phi : H \rightarrow Z \times Z \times Z$ such that $\Phi(h)=(x,y,z)$ for
${h = \left(
\begin{array}{@{}ccc@{}}
1 & y & z \\
0 & 1 & x \\
0 & 0 & 1
\end{array}
\right) \in H}$.
Note that $\phi$ has length $O(\mathtt{log}~n)$.

Now we are ready to write the sentence $\psi$ describing $H$.
Let $n = \prod_{1 \le i \le m} p_i^{z_i}$ be the prime decomposition of $n$.
Then the sentence is $\psi \equiv \exists a \exists b [\psi_1 \wedge \ldots \wedge \psi_6]$ where $\psi_1$ says that $a,b$ have order dividing $n$
\[
\psi_1(a,b) \equiv \theta_n(a,1)~\wedge~\theta_n(b,1)
\]
$\psi_2$ says that $c$ has order $n$ (using the previous result)
\[
\begin{split}
\psi_2(a,b) \equiv~&\theta_n([a,b],1)~\wedge \\
&\exists c_1 \ldots \exists c_m \left[ \bigwedge_{1 \le i \le m} \left\{ \pi'_n(c_i;[a,b])~\wedge~\theta_{p_i^{z_i}}(c_i,1)~\wedge~\neg \theta_{p_i^{z_i-1}}(c_i,1) \right\} \right]
\end{split}
\]
$\psi_3$ says that $H'=Z=\langle c \rangle$ coincides with the set of commutators
\[
\begin{split}
\psi_3(a,b) \equiv~&\forall r \forall s \forall t \forall u \exists v \exists w~[r,s][t,u]=[v,w]~\wedge~\forall r \forall s \forall h [[r,s],h]=1~\wedge \\
&\forall z \{ \forall h~[z,h]=1 \rightarrow (\exists r \exists s~[r,s]=z~\wedge~\pi'_n([a,b],z)\}
\end{split}
\]
$\psi_4,\psi_5$ say that $\Phi(h)$ is a function $H \rightarrow Z \times Z \times Z$
\[
\begin{split}
\psi_4(a,b) \equiv~&\forall h \exists x \exists y \exists z~\phi(h,x,y,z;a,b)\\
\psi_5(a,b) \equiv~&\forall h \forall x_1 \forall x_2 \forall y_1 \forall y_2 \forall z_1 \forall z_2\\
&[\{\phi(h,x_1,y_1,z_1;a,b)~\wedge~\phi(h,x_2,y_2,z_2;a,b)\} \rightarrow\\ &\hspace{100pt}\{x_1=x_2~\wedge~y_1=y_2~\wedge~z_1=z_2\}]
\end{split}
\]
and $\psi_6$ says that $\Phi$ is surjective
\[
\psi_6(a,b) \equiv~\forall x \forall y \forall z \{\forall g~[x,g]=[y,g]=[z,g]=1 \rightarrow \exists h~\phi(h,x,y,z;a,b)\}.
\]

It can be easily seen that $\psi$ has length $O(\mathtt{log}~n)$.

Now let $G$ be a group satisfying $\psi$ with witnesses $a,b \in G$.
Then from $\psi_2,\psi_3$, the center $Z$ of $G$ is cyclic of order $n$ generated by $c=[a,b]$.
Since $\phi$ defines an surjective function $\Phi : G \rightarrow Z \times Z \times Z$ from $\psi_4,\psi_5,\psi_6$, $G$ has size at least $n^3$.
But since $a,b$ have order at most $n$ from $\psi_1$ and each element $g \in G$ can be written as a product of the form $g=uvz$ where $u \in \langle a \rangle$, $v \in \langle b \rangle$, $z \in Z$ from $\psi_4$, $G$ cannot have more than $n^3$ elements.
Hence $G$ has precisely $n^3$ elements, and has the form $G=\{a^{\alpha} b^{\beta} c^{\gamma}~|~\alpha, \beta, \gamma \in \ZZ_n \}$.

Since $G$ is class 2 nilpotent from $\psi_3$ and $c=[a,b]$, one can deduce the equation
\[
a^{\alpha_1} b^{\beta_1} c^{\gamma_1} \cdot a^{\alpha_2} b^{\beta_2} c^{\gamma_2} = a^{\alpha_1+\alpha_2} b^{\beta_1+\beta_2} c^{\gamma_1+\gamma_2-\alpha_2 \beta_1}
\]
and it determines the group uniquely up to isomorphism.
\end{proof}

The short presentation conjecture \cite{Babai:97} asks whether there exists a constant $C$ such that every finite group $G$ has a presentation of length $O({\mathtt {log}}^c~|G|)$. In analogy, we ask:

\begin{question} Is the class of finite groups polylogarithmically compressible (PLC)? Is it in fact logarithmically compressible? \end{question}


\def\cprime{$'$}

\end{document}